\documentclass[11pt,reqno]{amsart}

\usepackage[T1]{fontenc}
\usepackage[utf8]{inputenc}
\usepackage[english]{babel}
\usepackage[textwidth=15.6cm,textheight=23cm,centering]{geometry}
\usepackage{microtype}

\usepackage{amsmath,amssymb,amsfonts,amsthm,mathtools,mathrsfs,bm}
\usepackage{enumitem}
\usepackage{booktabs}

\usepackage{graphicx}
\usepackage{subcaption}
\usepackage{tikz}
\usepackage{tikz-cd}
\usetikzlibrary{arrows.meta,calc,positioning}

\usepackage{xcolor}
\usepackage[
    colorlinks=true,
    linkcolor=blue!50!black,
    citecolor=blue!50!black,
    urlcolor=blue!50!black
]{hyperref}
\usepackage[nameinlink,capitalize,noabbrev]{cleveref}

\allowdisplaybreaks
\numberwithin{equation}{section}

\newif\ifdraft
\draftfalse

\ifdraft
    \newcommand{\todo}[1]{\textcolor{red}{[TODO: #1]}}
    \newcommand{\note}[1]{\textcolor{blue}{[NOTE: #1]}}
\else
    \newcommand{\todo}[1]{}
    \newcommand{\note}[1]{}
\fi

\theoremstyle{plain}
\newtheorem{theorem}{Theorem}[section]
\newtheorem{proposition}[theorem]{Proposition}
\newtheorem{lemma}[theorem]{Lemma}
\newtheorem{corollary}[theorem]{Corollary}

\theoremstyle{definition}

\theoremstyle{remark}
\newtheorem{remark}[theorem]{Remark}

\crefname{equation}{equation}{equations}
\Crefname{equation}{Equation}{Equations}
\crefname{section}{section}{sections}
\Crefname{section}{Section}{Sections}
\crefname{subsection}{subsection}{subsections}
\Crefname{subsection}{Subsection}{Subsections}
\crefname{theorem}{theorem}{theorems}
\Crefname{theorem}{Theorem}{Theorems}
\crefname{proposition}{proposition}{propositions}
\Crefname{proposition}{Proposition}{Propositions}
\crefname{lemma}{lemma}{lemmas}
\Crefname{lemma}{Lemma}{Lemmas}
\crefname{corollary}{corollary}{corollaries}
\Crefname{corollary}{Corollary}{Corollaries}
\crefname{claim}{claim}{claims}
\Crefname{claim}{Claim}{Claims}
\crefname{conjecture}{conjecture}{conjectures}
\Crefname{conjecture}{Conjecture}{Conjectures}
\crefname{question}{question}{questions}
\Crefname{question}{Question}{Questions}
\crefname{definition}{definition}{definitions}
\Crefname{definition}{Definition}{Definitions}
\crefname{notation}{notation}{notations}
\Crefname{notation}{Notation}{Notations}
\crefname{example}{example}{examples}
\Crefname{example}{Example}{Examples}
\crefname{exercise}{exercise}{exercises}
\Crefname{exercise}{Exercise}{Exercises}
\crefname{remark}{remark}{remarks}
\Crefname{remark}{Remark}{Remarks}

\newcommand{\R}{\mathbb{R}}
\newcommand{\Z}{\mathbb{Z}}

\renewcommand{\epsilon}{\varepsilon}
\renewcommand{\phi}{\varphi}

\renewcommand{\le}{\leqslant}
\renewcommand{\ge}{\geqslant}

\DeclareMathOperator{\GL}{GL}

\DeclareMathOperator{\Span}{span}

\DeclareMathOperator{\vol}{vol}

\DeclarePairedDelimiter{\abs}{\lvert}{\rvert}
\DeclarePairedDelimiter{\norm}{\lVert}{\rVert}
\DeclarePairedDelimiter{\angles}{\langle}{\rangle}
\DeclarePairedDelimiter{\floor}{\lfloor}{\rfloor}

\DeclarePairedDelimiter{\set}{\lbrace}{\rbrace}
\DeclarePairedDelimiterX{\inner}[2]{\langle}{\rangle}{{#1},\,{#2}}

\newcommand{\defeq}{\coloneqq}

\newcommand{\RomanNumeralCaps}[1]{\MakeUppercase{\romannumeral #1}}

\newcommand{\dual}[1]{{#1}^{\ast}}
\newcommand{\polar}[1]{{#1}^{\circ}}
\newcommand{\transpose}{^{\mathsf{T}}}
\newcommand{\succmin}[3]{\lambda_{#1}\!\left(#2,#3\right)}

\title[Improved bounds for a discrete John-type theorem]{Improved bounds for a discrete John-type theorem}

\author{Danila Solunov}
\address{HSE University, Pokrovsky Boulevard 11, Moscow, Russia 109028}
\email{dsolunov@gmail.com}


\begin{document}
 
\begin{abstract}
    Tao and Vu introduced a discrete analogue of John's theorem in which convex progressions are approximated by generalized arithmetic progressions. In the covering version of this problem, one asks for a small GAP containing all lattice points of a given origin-symmetric convex body. We prove that every such convex progression in dimension $n$ admits an infinitely proper GAP cover whose size is within a factor $O(n)^{2n}$ of the cardinality of the original set, improving the previously known factor $O(n)^{3n}$. We also show that a loss of order $\Omega(n)^n$ is unavoidable for infinitely proper GAP covers.
\end{abstract}

\maketitle

\section{Introduction} \label{sec:introduction}

John's theorem is one of the basic approximation results in convex geometry. It says that every centrally symmetric convex body can be replaced, up to a dimension-dependent loss, by a much simpler object --- an ellipsoid. More precisely, if $K \subset \R^n$ is an origin-symmetric convex body, then there exists an ellipsoid $\mathcal{E} \subset \R^n$ such that
\begin{gather*}
    \mathcal{E} \subseteq K \subseteq \sqrt{n}\,\mathcal{E}.
\end{gather*}

The discrete counterpart of this problem was introduced by Tao and Vu
\cite{TaoVu2006AdditiveCombinatorics,TaoVu2008JohnType}. In the discrete setting the goal is to approximate a convex progression, that is, the set of lattice points inside a convex body. The role of the approximating objects is played by symmetric generalized arithmetic progressions (GAPs), namely sets of the form
\begin{gather*}
    P = \set*{Az : z \in \Z^d, \, \abs{z_i} \le b_i, \, 1 \le i \le d},
\end{gather*}
where the columns of $A$ are $d$ lattice vectors and $b \in \R_{> 0}^d$. In the present paper, we focus on approximations by infinitely proper GAPs; in the representation above, this is equivalent to the columns of $A$ being $\Z$-linearly independent. Unlike general convex progressions, GAPs have a much simpler and more explicit structure, which makes them easier to control and compute. The discrete John problem asks how efficiently a convex progression can be replaced by such an object.

For a GAP $P$ defined as above and a number $t > 0$, let $P_t$ denote the GAP obtained by keeping the same generator matrix $A$ and replacing each parameter $b_i$ by $t b_i$. With this notation, Tao and Vu proved the following discrete analogue of John's theorem.

\begin{theorem}[{\cite[Theorem~1.6]{TaoVu2008JohnType}}] \label{thm:TaoVuJohn}
    Let $K \subset \R^n$ be an origin-symmetric convex body and $\Gamma \subset \R^n$ be a full-rank lattice. Then there exists an infinitely proper GAP $P$ in $\Gamma$ such that
    \begin{enumerate}
        \item $P \subseteq K \cap \Gamma \subseteq P_{O(n)^{3n/2}}$,
        \item $\abs{K \cap \Gamma} \le O(n)^{7n/2} \abs{P}$.
    \end{enumerate}
\end{theorem}

Berg and Henk substantially improved the bounds of Tao and Vu. They showed that the two conclusions above can be significantly improved separately.    

\begin{theorem}[{\cite[Theorem~1.1]{BergHenk2019DiscreteJohn}}] \label{thm:BergHenkJohnSeparately}
    Let $K \subset \R^n$ be an origin-symmetric convex body and $\Gamma \subset \R^n$ be a full-rank lattice. Then the following statements hold.
    \begin{enumerate}
        \item There exists an infinitely proper GAP $P$ in $\Gamma$ such that
            \begin{gather*}
                P \subseteq K \cap \Gamma \subseteq P_{n^{O(\log n)}}.
            \end{gather*}
        \item There exists an infinitely proper GAP $P$ in $\Gamma$ such that $P \subseteq K \cap \Gamma$ and
            \begin{gather*}
                \abs{K \cap \Gamma} \le O(n)^n \abs{P}.
            \end{gather*}
    \end{enumerate}
\end{theorem}
In general, the two GAPs in \Cref{thm:BergHenkJohnSeparately} need not be the same. Berg and Henk also proved a simultaneous version, which slightly improves the result of \Cref{thm:TaoVuJohn}.

Instead of seeking a proportional GAP approximation of a convex progression, one can
consider two weaker one-sided problems. The first problem asks for the smallest dimension-dependent constant $C_{\mathrm{in}}(n)$ such that, for every origin-symmetric convex body $K \subset \R^n$ and every full-rank lattice $\Gamma \subset \R^n$, there exists an infinitely proper GAP $P_{\mathrm{in}} \subseteq K \cap \Gamma$ such that
\begin{gather*}
    \abs{K \cap \Gamma} \le C_{\mathrm{in}}(n) \, \abs{P_{\mathrm{in}}}.
\end{gather*} 
The second part of \Cref{thm:BergHenkJohnSeparately} gives the upper bound $C_{\mathrm{in}}(n) \le O(n)^n$. Berg and Henk also proved a lower bound of the form $C_{\mathrm{in}}(n) \ge (2^n + 1) / 3$, showing that an exponential loss is unavoidable in this problem.

The second one-sided problem is the covering problem. Here one asks for the smallest dimension-dependent constant $C_{\mathrm{out}}(n)$ such that, for every origin-symmetric convex body $K \subset \R^n$ and every full-rank lattice $\Gamma \subset \R^n$, there exists an infinitely proper GAP $P_{\mathrm{out}} \supseteq K \cap \Gamma$ such that
\begin{gather*}
    \abs{P_{\mathrm{out}}} \le C_{\mathrm{out}}(n) \, \abs{K \cap \Gamma}.
\end{gather*}
This is the problem studied in the present paper.

The proportional inclusion results presented in \Cref{thm:TaoVuJohn,thm:BergHenkJohnSeparately} give immediate upper bounds for $C_{\mathrm{out}}(n)$. Indeed, if $P \subseteq K \cap \Gamma \subseteq P_t$ for some $t \ge 1$, then $P_t$ is a GAP containing $K \cap \Gamma$. Moreover, since $P$ has rank at most $n$, we have 
\begin{gather*}
    \abs{P_t} \le O(t)^n \abs{P} \le O(t)^n \abs{K \cap \Gamma}.
\end{gather*}
Therefore, \Cref{thm:TaoVuJohn} gives $C_{\mathrm{out}}(n) \le n^{O(n^2)}$, while the first part of \Cref{thm:BergHenkJohnSeparately} gives the improved bound $C_{\mathrm{out}}(n) \le n^{O(n \log n)}$.

The covering problem was recently studied directly by van Hintum and Keevash, who obtained the following result.

\begin{theorem}[{\cite[Theorem~1.1]{vanHintumKeevash2024DiscreteJohn}}]
    Let $K \subset \R^n$ be an origin-symmetric convex body and $\Gamma \subset \R^n$ be a full-rank lattice. Then there exists an infinitely proper GAP $P$ in $\Gamma$ such that
    \begin{enumerate}
        \item $K \cap \Gamma \subseteq P$,
        \item $\abs{P} \le O(n)^{3 n} \abs{K \cap \Gamma}$.
    \end{enumerate}
\end{theorem}

Equivalently, this gives $C_{\mathrm{out}}(n) \le O(n)^{3n}$. The main result of the present paper improves the bound of van Hintum and Keevash from $O(n)^{3n}$ to $O(n)^{2n}$.

\begin{theorem} \label{thm:GAP_cover}
    Let $K \subset \R^n$ be an origin-symmetric convex body and $\Gamma \subset \R^n$ be a full-rank lattice. Then there exists an infinitely proper GAP $P$ in $\Gamma$ such that
    \begin{enumerate}
        \item $K \cap \Gamma \subseteq P$,
        \item $\abs{P} \le O(n)^{2n} \abs{K \cap \Gamma}$.
    \end{enumerate}
\end{theorem}

We also prove a lower bound, showing that $C_{\mathrm{out}}(n) \ge \Omega(n)^n$.

\begin{theorem} \label{thm:GAP_lower_bound}
    There exists an absolute constant $c > 0$ such that for every $n \in \Z_{>0}$ the following holds. For every full-rank lattice $\Gamma \subset \R^n$ there exists an origin-symmetric convex body $K \subset \R^n$ such that every infinitely proper GAP $P$ in $\Gamma$ with $K \cap \Gamma \subseteq P$ satisfies
    \begin{gather*}
        \abs{P} \ge (c n)^{n} \abs{K \cap \Gamma}.
    \end{gather*}
\end{theorem}

Thus the best known bounds on $C_{\mathrm{out}}(n)$ are
\begin{gather*}
    \Omega(n)^n \le C_{\mathrm{out}}(n) \le O(n)^{2n}.
\end{gather*}

\section{Notation and Preliminaries}\label{sec:preliminaries}

\subsection{Convex bodies}
Throughout the paper, all convex bodies are assumed to be compact convex sets with nonempty interior. A convex body $K$ is called origin-symmetric if $K = -K$. For such a body we write 
\begin{gather*}
    \norm{x}_K = \inf \set*{t > 0 : x \in tK}.
\end{gather*}
We also write $B_p^n$ for the unit ball of $\ell_p^n$.

\subsection{Lattices}
A lattice $\Gamma \subset \R^n$ of rank $d$ is a set of integer linear combinations of $d$ linearly independent vectors $B = (b_1, \ldots, b_d)$, 
\begin{gather*}
    \Gamma = \set*{\sum_{i = 1}^d z_i b_i : z_i \in \Z}.
\end{gather*}
The vectors $b_1, \ldots, b_d$ are called a basis of $\Gamma$. The determinant of the lattice $\Gamma$ is given by $\det{\Gamma} = \sqrt{\det{B\transpose B}}$ and is independent of the choice of the basis of $\Gamma$.

For a full-rank lattice $\Gamma \subset \R^n$, the dual lattice is defined by
\begin{gather*}
    \dual{\Gamma} = \set*{y \in \R^n : \inner{x}{y} \in \Z \quad \forall x \in \Gamma}.
\end{gather*}
If $b_1, \ldots, b_n$ is a basis of $\Gamma$, then there exists a unique dual basis $b_1^\ast, \ldots, b_n^\ast$ of $\dual{\Gamma}$ satisfying $\inner{b_i}{b_j^\ast} = \delta_{i, j}$. Also $\det{\dual{\Gamma}} = (\det{\Gamma})^{-1}$.

\subsection{Generalized arithmetic progressions}

A symmetric generalized arithmetic progression, or GAP, in a lattice $\Gamma \subset \R^n$ is a set of lattice points of the form
\begin{gather*}
    P = \set*{Az : z \in \Z^d, \, \abs{z_i} \le b_i, \, 1 \le i \le d},
\end{gather*}
where $A \in \R^{n \times d}$ is a matrix with columns $a_1, \ldots, a_d \in \Gamma$ and $b = (b_1, \ldots, b_d) \in \R_{>0}^d$. The number $d$ is called the rank of $P$.

The GAP is called proper if the map $z \mapsto Az$ is injective on the box
\begin{gather*}
    \set*{z \in \Z^d : \, \abs{z_i} \le b_i, \, 1 \le i \le d}.
\end{gather*}
For $t > 0$, we denote by $P_t$ the GAP obtained by replacing each parameter $b_i$ by $t b_i$ and keeping the same generator matrix $A$. We say that $P$ is infinitely proper if $P_t$ is proper for every $t > 0$. Equivalently, the columns of the matrix $A$ are $\Z$-linearly independent.

The size of an infinitely proper GAP $P$ is defined by
\begin{gather*}
    \abs{P} = \prod_{i = 1}^d (2 \floor{b_i} + 1),
\end{gather*}
where $\floor{\cdot}$ denotes the floor function.

\subsection{Successive minima}

For a convex body $K$ and a lattice $\Gamma$, we define the $i$-th successive minimum $\succmin{i}{K}{\Gamma}$ as the least $\lambda > 0$ for which $\lambda K$ contains $i$ linearly independent vectors of $\Gamma$. We shall use the following standard facts about successive minima.

\begin{proposition}[Minkowski's second theorem {\cite[Chapter~\RomanNumeralCaps{8}]{Cassels1959GeometryNumbers}}] \label{thm:minkowski_second}
    Let $K \subset \R^n$ be an origin-symmetric convex body and $\Gamma \subset \R^n$ be a full-rank lattice. Then
    \begin{gather*}
        \vol{K} \prod_{i = 1}^n \lambda_i(K, \Gamma) \le 2^n \det{\Gamma}.
    \end{gather*}
\end{proposition}

Although a lattice basis consisting of vectors that attain the successive minima does not always exist, it is always possible to construct a basis whose vector norms are close to these optimal values.

\begin{proposition}[Mahler basis {\cite[Chapter~\RomanNumeralCaps{8}]{Cassels1959GeometryNumbers}}] \label{thm:mahler_basis}
    Let $K \subset \R^n$ be an origin-symmetric convex body and let $\Gamma \subset \R^n$ be a full-rank lattice. Then there exists a basis $b_1, \ldots, b_n$ of $\Gamma$ such that
    \begin{gather*}
        \norm{b_i}_K \le \max \set*{1, \frac{i}{2}} \lambda_i(K, \Gamma).
    \end{gather*}
\end{proposition}

\subsection{Lattice point enumerator}
\begin{proposition}[{\cite[Chapter~\RomanNumeralCaps{3}]{Cassels1959GeometryNumbers}}]
    Let $K \subset \R^n$ be an origin-symmetric convex body and $\Gamma \subset \R^n$ be a full-rank lattice. If, for some $m \in \Z_{\ge 0}$,
    \begin{gather*}
        \vol{K} \ge m 2^n \det{\Gamma},
    \end{gather*}
    then $K$ contains at least $m$ pairs of points $\pm u_j$ for $1 \le j \le m$, which are distinct from each other and from the origin.
\end{proposition}

A simple consequence of this theorem is the following estimate on the lattice point enumerator.
\begin{corollary} \label{col:van_der_corput}
    Let $K \subset \R^n$ be an origin-symmetric convex body and $\Gamma \subset \R^n$ be a full-rank lattice. Then
    \begin{gather*}
        \vol{K} < 2^n \det \Gamma \, \abs{K \cap \Gamma}.
    \end{gather*}
\end{corollary}

\begin{lemma}[Lattice point enumerator for a cross-polytope] \label{lem:enumerator_cross_polytope}
    For any $n, m \in \Z_{>0}$ the following formula holds.
    \begin{gather*}
        \abs{m B_1^n \cap \Z^n} = \sum_{k = 0}^n 2^k \binom{n}{k} \binom{m}{k} 
    \end{gather*}
\end{lemma}
\begin{proof}
    Count lattice points according to the number $k$ of nonzero coordinates. For each choice of the support in $\binom{n}{k}$ ways and each choice of the signs in $2^k$ ways, the absolute values of the nonzero coordinates are positive integers $u_1, \ldots, u_k$ with $u_1 + \cdots + u_k \le m$. The number of such $k$-tuples is $\binom{m}{k}$ according to the stars and bars argument. Summing over all values of $k$ gives the formula. 
\end{proof}

\subsection{Polar body and Mahler volume}
For a convex body $K \subset \R^n$ with the origin in its interior, the polar body is defined by
\begin{gather*}
    \polar{K} = \set{y \in \R^n : \angles{x, y} \le 1 \, \forall x \in K}.
\end{gather*}
The polar body is always convex, and if $K$ is origin-symmetric, then $\polar{K}$ is origin-symmetric again. The product
\begin{gather*}
    \nu(K) = \vol{K} \vol{\polar{K}}
\end{gather*}
is called the Mahler volume of $K$. We shall only need the following lower bound on $\nu(K)$.

\begin{proposition}[{\cite[Corollary~1.6]{Kuperberg2008Mahler}}] \label{thm:mahler_product_lower_bound}
    For every origin-symmetric convex body $K \subset \R^n$, one has
    \begin{gather*}
        \nu(K) = \vol{K} \vol{\polar{K}} \ge \frac{\pi^n}{n!}.
    \end{gather*}
\end{proposition}

\section{Auxiliary results}\label{sec:auxiliary-results}
The proof of the main result relies on a relation linking the lattice point enumerator $\abs{K\cap\Gamma}$ with the product of the successive minima associated with the pair $(\polar{K}, \dual{\Gamma})$. To state this relation conveniently, define

\begin{gather*}
    \mathfrak{D}_n = \sup_{K, \Gamma} \frac{\prod_{i = 1}^n \succmin{i}{\polar{K}}{\dual{\Gamma}}}{\abs{K \cap \Gamma}},
\end{gather*}
where the supremum is taken over all origin-symmetric convex bodies $K \subset \R^n$ and all full-rank lattices $\Gamma \subset \R^n$.

\begin{proposition} \label{prop:polar_enumerator_constant_bound}
    For every $n \ge 1$,
    \begin{gather*}
        \left(\frac{1}{2}\right)^n n! \le \mathfrak{D}_n \le \left(\frac{4}{\pi}\right)^n n!.
    \end{gather*}
\end{proposition}
\begin{proof}[Proof of the upper bound]
    \Cref{thm:minkowski_second} applied to $\polar{K}$ and $\dual{\Gamma}$ gives
    \begin{gather*}
        \vol{\polar{K}} \prod_{i = 1}^n \succmin{i}{\polar{K}}{\dual{\Gamma}} \le 2^n \det{\dual{\Gamma}} = \frac{2^n}{\det{\Gamma}}.
    \end{gather*}
    By \Cref{thm:mahler_product_lower_bound}
    \begin{gather*}
        \frac{1}{\vol{\polar{K}}} \le \frac{n! \vol{K}}{\pi^n}.
    \end{gather*}
    Thus,
    \begin{gather*}
        \prod_{i = 1}^n \succmin{i}{\polar{K}}{\dual{\Gamma}} \le \frac{2^n}{\vol{\polar{K}} \det{\Gamma}} \le \left(\frac{2}{\pi}\right)^n n! \frac{\vol{K}}{\det \Gamma}.
    \end{gather*}
    Finally, \Cref{col:van_der_corput} gives
    \begin{gather*}
        \frac{\vol{K}}{\det{\Gamma}} \le 2^n \abs{K \cap \Gamma}.
    \end{gather*}
    The combination of the last two inequalities gives the stated upper bound.
\end{proof}
\begin{proof}[Proof of the lower bound]
    To prove the lower bound, take $K_m = m B_1^n$, with $m$ a positive integer. Then $K_m^{\circ} = \frac{1}{m} B_{\infty}^n$. For $\Gamma = \Z^n$, since $\dual{(\Z^n)} = \Z^n$, we have
    \begin{gather*}
        \succmin{i}{K_m^{\circ}}{\dual{(\Z^n)}} = m, \qquad i = 1, \ldots, n.
    \end{gather*}
    On the other hand, by \Cref{lem:enumerator_cross_polytope} we have
    \begin{gather*}
        \abs{K_m \cap \Z^n} = \abs{m B_1^n \cap \Z^n} = \sum_{k = 0}^n 2^k \binom{n}{k} \binom{m}{k}.
    \end{gather*}
    Taking the supremum over the pairs $(K_m, \Z^n)$ gives the following lower bound on $\mathfrak{D}_n$.
    \begin{gather*}
        \mathfrak{D}_n \ge \sup_{m \in \Z_{>0}} \frac{\prod_{i = 1}^n \succmin{i}{K_m^\circ}{(\Z^n)^\ast}}{\abs{K_m \cap \Z^n}} = \sup_{m \in \Z_{>0}} \frac{m^n}{\sum_{k = 0}^n 2^k \binom{n}{k} \binom{m}{k}} \ge \lim_{m \to \infty} \frac{m^n}{\sum_{k = 0}^n 2^k \binom{n}{k} \binom{m}{k}} =\\= \lim_{m \to \infty} \frac{m^n}{2^n \binom{m}{n}} = \frac{n!}{2^n} \cdot \lim_{m \to \infty} \frac{m^n}{m (m - 1) \ldots (m - n + 1)} = \frac{n!}{2^n}
    \end{gather*}
\end{proof}

\section{Proof of the upper bound}\label{sec:main-proof}

\begin{lemma} \label{lem:dual_successive_minima}
    If $\Span(K \cap \Gamma) = \R^n$, then every non-zero vector $a \in \dual{\Gamma}$ satisfies
    \begin{gather*}
        \norm{a}_{\polar{K}} \ge 1.
    \end{gather*}
    In particular,
    \begin{gather*}
        1 \le \succmin{1}{\polar{K}}{\dual{\Gamma}} \le \ldots \le \succmin{n}{\polar{K}}{\dual{\Gamma}}.
    \end{gather*}
\end{lemma}
\begin{proof}
    Let $a \in \dual{\Gamma} \setminus \set{0}$. Since $\Span(K \cap \Gamma) = \R^n$, there exists $x \in K \cap \Gamma$ with $\inner{x}{a} \neq 0$. Since $a \in \dual{\Gamma}$, we have $\inner{x}{a} \in \Z$, and therefore $\abs{\inner{x}{a}} \ge 1$. Consequently,
    \begin{gather*}
        \norm{a}_{\polar{K}} = \sup_{y \in K} \abs{\inner{y}{a}} \ge \abs{\inner{x}{a}} \ge 1.
    \end{gather*}
    The inequalities for the successive minima follow immediately.
\end{proof}

\begin{proof}[Proof of \Cref{thm:GAP_cover}]
    By passing to a subspace if necessary, we may assume that $K \cap \Gamma$ is full-dimensional.
    By \Cref{thm:mahler_basis} applied to $\polar{K}$ and $\dual{\Gamma}$, there exists a basis $a_1, \ldots, a_n$ of $\dual{\Gamma}$ such that
    \begin{gather*}
        \norm{a_i}_{\polar{K}} \le \gamma_i, \qquad \gamma_i = \max\set*{1, \frac{i}{2}} \succmin{i}{\polar{K}}{\dual{\Gamma}}.
    \end{gather*}
    Let $b_1, \ldots, b_n$ be the basis of $\Gamma$ dual to $a_1, \ldots, a_n$, so that $\inner{a_i}{b_j} = \delta_{i, j}$ for all $i, j = 1, \ldots, n$. 
    Let $x \in K \cap \Gamma$. Using the dual basis, we may write
    \begin{gather*}
        x = \sum_{i = 1}^n \inner{x}{a_i} b_i.
    \end{gather*}
    Since $a_i \in \dual{\Gamma}$, each coefficient $\inner{x}{a_i}$ is an integer. Also, $a_i \in \gamma_i \polar{K}$ and $x \in K$, hence
    \begin{gather*}
        \abs{\inner{x}{a_i}} \le \gamma_i, \qquad 1 \le i \le n.
    \end{gather*}
    It follows that 
    \begin{gather*}
        x \in P \defeq \left\{ \sum_{i = 1}^n z_i b_i : z_i \in \Z, \, \abs{z_i} \le \gamma_i \right\}.
    \end{gather*}
    Thus, $P$ is a GAP containing all points of $K \cap \Gamma$. It is infinitely proper, because $b_1, \ldots, b_n$ form a basis of $\Gamma$. Since $\Span(K \cap \Gamma) = \R^n$, \Cref{lem:dual_successive_minima} gives $\gamma_i \ge 1$ for all $i$. Therefore the size of the GAP $P$ can be bounded by
    \begin{gather*}
        \abs{P} = \prod_{i = 1}^n \left(2 \floor{\gamma_i} + 1\right) \le 3^n \prod_{i = 1}^n \gamma_i = 3^n \cdot \frac{n!}{2^{n - 1}} \prod_{i = 1}^n \succmin{i}{\polar{K}}{\dual{\Gamma}}.
    \end{gather*}
    By \Cref{prop:polar_enumerator_constant_bound}, 
    \begin{gather*}
        \prod_{i = 1}^n \succmin{i}{\polar{K}}{\dual{\Gamma}} \le \mathfrak{D}_n \abs{K \cap \Gamma} \le \left(\frac{4}{\pi}\right)^n n! \abs{K \cap \Gamma}.
    \end{gather*}
    Consequently,
    \begin{gather*}
        \abs{P} \le 2 \left(\frac{6}{\pi}\right)^n n!^2 \abs{K \cap \Gamma} = O(n)^{2n} \abs{K \cap \Gamma},
    \end{gather*}
    where the last estimate follows from Stirling's formula.
\end{proof}
\begin{remark}
    The loss of $O(n)^{2n}$ consists of two independent factors. The first one comes from Mahler's basis theorem. More precisely, one can choose a basis $u_1, \ldots, u_n$ of $\Gamma$ such that
    \begin{gather*}
        \prod_{i = 1}^n \norm{u_i}_K
        \le
        \frac{n!}{2^{\,n-1}}
        \prod_{i = 1}^n \succmin{i}{K}{\Gamma}.
    \end{gather*}
    Thus any improvement in the loss between the minimal possible product of the lengths of basis vectors and the product of the successive minima would directly improve the final estimate. 
    
    The second factor comes from passing from the successive minima of the polar body to the lattice point enumerator. By \Cref{prop:polar_enumerator_constant_bound}, this step requires a loss of order $\Omega(n)^n$ in general, and therefore cannot be asymptotically improved within this method. 
\end{remark}

\section{Proof of the lower bound}

\begin{proof}[Proof of \Cref{thm:GAP_lower_bound}]
    First consider the case $\Gamma = \Z^n$ and take $K_m = m B_1^n$ with $m \in \Z_{>0}$. Let
    \begin{gather*}
        P_m = \set*{A z : z \in \Z^d, \abs{z_i} \le b_i, \, 1 \le i \le d}
    \end{gather*}
    be an infinitely proper GAP containing $K_m \cap \Z^n$. The columns of $A$ are $\Z$-linearly independent vectors of $\Z^n$. Moreover, since $e_1, \ldots, e_n \in K_m \cap \Z^n \subseteq P_m$, these columns generate the whole lattice $\Z^n$. Hence they form a basis of $\Z^n$, and therefore $A \in \GL_n(\Z)$. Since $\pm m e_j \in K_m \cap \Z^n$, we have
    \begin{gather*}
        \abs{(A^{-1} m e_j)_i} \le b_i \qquad \forall 1 \le i, j \le n.
    \end{gather*}
    Since $A \in \GL_n(\Z)$, its inverse $A^{-1}$ also belongs to $\GL_n(\Z)$. Hence, for each $1 \le i \le n$, there exists $1 \le j \le n$ such that $\abs{(A^{-1})_{i, j}} \ge 1$. Therefore,
    \begin{gather*}
        b_i \ge \abs{(A^{-1} m e_j)_i} = m \, \abs{(A^{-1})_{i, j}} \ge m.
    \end{gather*}
    It follows that
    \begin{gather*}
        \abs{P_m} = \prod_{i = 1}^n (2 \floor{b_i} + 1) \ge (2m + 1)^n.
    \end{gather*}
    On the other hand, by \Cref{lem:enumerator_cross_polytope},
    \begin{gather*}
        \abs{K_m \cap \Z^n} = \abs{m B_1^n \cap \Z^n} = \sum_{k = 0}^n 2^k \binom{n}{k} \binom{m}{k}.
    \end{gather*}
    Hence,
    \begin{gather*}
        \lim_{m \to \infty} \frac{(2 m + 1)^n}{\abs{K_m \cap \Z^n}} = \lim_{m \to \infty} \frac{(2m + 1)^n}{\sum_{k = 0}^n 2^k \binom{n}{k} \binom{m}{k}} = \lim_{m \to \infty} \frac{(2m + 1)^n}{2^n \binom{m}{n}} = n!.
    \end{gather*}
    For a fixed $n$, choose $m$ so large that this ratio is at least $\frac{n!}{2}$. Since $n! \ge \left(\frac{n}{e}\right)^n$, we obtain
    \begin{gather*}
        \abs{P_m} \ge \frac{n!}{2} \abs{K_m \cap \Z^n} \ge \left(\frac{n}{2e}\right)^n \abs{K_m \cap \Z^n}.
    \end{gather*}
    This proves the statement for the lattice $\Z^n$, for instance with $c = \frac{1}{2e}$. 

    Now let $\Gamma \subset \R^n$ be an arbitrary full-rank lattice and choose a linear isomorphism $T$ such that $T \Z^n = \Gamma$. Put $K = T K_m$, where $m$ is chosen as above. Then $K$ is an origin-symmetric convex body and $\abs{K \cap \Gamma} = \abs{K_m \cap \Z^n}$. If an infinitely proper GAP $P$ in $\Gamma$ contains $K \cap \Gamma$, then $T^{-1} P$ is an infinitely proper GAP in $\Z^n$ containing $K_m \cap \Z^n$ and has the same size as $P$. Therefore, the already proved statement for the lattice $\Z^n$ gives
    \begin{gather*}
        \abs{P} = \abs{T^{-1} P} \ge \left(\frac{n}{2e}\right)^n \abs{K_m \cap \Z^n} = \left(\frac{n}{2e}\right)^n \abs{K \cap \Gamma}.
    \end{gather*}
\end{proof}

\bibliographystyle{amsalpha}
\bibliography{refs}

\end{document}